\documentclass[12pt]{article}
\usepackage[
  a4paper,
  top=2.6cm,
  bottom=2.8cm,
  left=2.8cm,
  right=2.8cm,
  headheight=14pt,
  headsep=0.7cm,
  footskip=1.2cm,
  heightrounded
]{geometry}

\usepackage{amsmath,amsthm}

\usepackage[T1]{fontenc}
\usepackage{newtxtext}
\usepackage{newtxmath}

\usepackage{microtype}

\usepackage{setspace}
\usepackage{float}

\usepackage{tikz}
\usetikzlibrary{backgrounds}
\usetikzlibrary{arrows}
\usetikzlibrary{shapes,shapes.geometric,shapes.misc}
\pgfdeclarelayer{edgelayer}
\pgfdeclarelayer{nodelayer}
\pgfsetlayers{background,edgelayer,nodelayer,main}

\usetikzlibrary{bbox}

\usepackage{mathtools}
\usepackage{enumitem}
\usepackage[hidelinks]{hyperref}
\usepackage{microtype}
\usepackage{xcolor}
\usepackage[abbrev,msc-links,nobysame]{amsrefs}  

\usepackage{doi}

\hypersetup{
    colorlinks=true,
    linkcolor=blue,
    citecolor=blue,
    urlcolor=blue
}

\tikzstyle{none}=[inner sep=0mm]
\tikzstyle{bluenode}=[fill=white, draw=black, shape=circle, minimum size=0.2cm, inner sep=1pt, scale=0.75]
\tikzstyle{blacknode}=[fill=black, draw=black, shape=circle, minimum size=0.2cm, inner sep=1pt]
\tikzstyle{pinknode}=[fill={rgb,255: red,255; green,191; blue,191}, draw=black, shape=circle, minimum size=0.2cm, inner sep=1pt, scale=0.65]
\tikzstyle{rednode}=[fill={rgb,255: red,244; green,0; blue,0}, draw=black, shape=circle, minimum size=0.2cm, inner sep=1pt, scale=0.75]
\tikzstyle{whitenode}=[fill={rgb,255: red,245; green,245; blue,245}, draw=black, shape=circle, minimum size=0.2cm, inner sep=1pt, scale=0.65]
\tikzstyle{springnode}=[fill={rgb,255: red,44; green,218; blue,154}, draw=black, shape=circle, minimum size=0.2cm, inner sep=1pt, scale=0.65]
\tikzstyle{orangenode}=[fill={rgb,255: red,255; green,128; blue,0}, draw=black, shape=circle, minimum size=0.15cm, inner sep=0pt]
\tikzstyle{whilenode2}=[fill=white, draw=white, shape=circle]
\tikzstyle{testnode}=[fill={rgb,255: red,255; green,191; blue,191}, draw=black, shape=circle, minimum size=2cm, inner sep=0.3pt, line width=0.45mm]
\tikzstyle{testnode2}=[fill=none, draw=black, shape=circle, minimum size=1cm, inner sep=0.1pt, line width=0.45mm]
\tikzstyle{testnode3}=[fill=none, draw=black, shape=circle, minimum size=4cm, inner sep=0.1pt, line width=0.45mm]
\tikzstyle{square}=[draw=black, shape=rectangle, minimum size=0.3cm, inner sep=1pt, fill=white]
\tikzstyle{dot}=[fill=black, draw=black, shape=circle, minimum size=0.04cm, inner sep=0pt]
\tikzstyle{yellownode}=[fill=yellow, draw=black, shape=circle, minimum size=0.2cm, inner sep=0pt, scale=0.75]
\tikzstyle{redsquare}=[fill=white, draw=black, shape=rectangle, draw=red]
\tikzstyle{bluesquare}=[fill=white, draw=black, shape=rectangle, draw=blue]

\tikzstyle{blackedge}=[-, draw=black, fill=none, line width=0.2mm]

\tikzstyle{balck_dash}=[-, dash pattern=on 0.2mm off 0.2mm]
\tikzstyle{black_thick}=[-, draw=black, line width=0.45mm, fill=none]
\tikzstyle{greenedge}=[-, draw={rgb,255: red,9; green,122; blue,43}]
\tikzstyle{pureedge}=[-, draw={rgb,255: red,203; green,0; blue,203}, line width=0.5mm]

\tikzstyle{rededge}=[-, draw=red, line width=0.3mm]
\tikzstyle{red_dash}=[-, dash pattern=on 0.2mm off 0.2mm, draw=red]
\tikzstyle{rededge_thick}=[-, line width=0.2mm, draw=red]
\tikzstyle{dashpure1}=[-, dashed, draw={rgb,255: red,128; green,0; blue,128}]
\tikzstyle{dashpure_thick}=[-, dashed, line width=0.5mm, draw={rgb,255: red,128; green,0; blue,128}]
\tikzstyle{blueedge}=[-, draw={rgb,255: red,18; green,94; blue,235}, line width=0.15mm]
\tikzstyle{blue_dash}=[-, draw={rgb,255: red,13; green,20; blue,109}, line width=0.15mm, dash pattern=on 0.2mm off 0.2mm]
\tikzstyle{dashedge}=[-, dash pattern=on 0.2mm off 0.2mm, draw={rgb,255: red,143; green,0; blue,50}]
\tikzstyle{blackedge2}=[-, draw=black]
\tikzstyle{black_directedge}=[draw=black, ->]
\tikzstyle{orangeedge}=[-, draw={rgb,255: red,255; green,128; blue,0}]
\tikzstyle{blueedge_thick}=[-, draw=blue, line width=0.5mm]
\tikzstyle{red_dash}=[-, draw=red, dash pattern=on 0.2mm off 0.2mm]
\tikzstyle{shadow_Lightgrey}=[-, fill={rgb,255: red,224; green,224; blue,224}, draw=none]
\tikzstyle{shadow_Lightgreen}=[-, draw=blue, fill={rgb,255: red,216; green,255; blue,242}, dash pattern=on 0.2mm off 0.2mm]
\tikzstyle{shadow}=[-, fill={rgb,255: red,130; green,202; blue,255}, draw=red]
\tikzstyle{shadow_blue}=[-, fill={rgb,255: red,130; green,202; blue,255}, draw=black]
\tikzstyle{shadowblack}=[-, fill={rgb,255: red,240; green,248; blue,255}, draw=black]
\tikzstyle{shadow_pure}=[-, fill={rgb,255: red,224; green,182; blue,250}, draw=none]
\tikzstyle{shadow_}=[-, fill={rgb,255: red,255; green,246; blue,112}, draw=none]
\tikzstyle{shadow_silver}=[-, draw=black, fill={rgb,255: red,186; green,186; blue,186}]
\tikzstyle{pickshing}=[-, draw={rgb,255: red,128; green,0; blue,128}, fill={rgb,255: red,255; green,191; blue,191}]
\tikzstyle{shadow_Lightgrey2}=[-, fill={rgb,255: red,193; green,193; blue,193}, draw=black]
\tikzstyle{shadow_silver2}=[-, draw=black, fill={rgb,255: red,186; green,186; blue,186}, line width=0.45mm]
\tikzstyle{yellowedge}=[-, draw={rgb,255: red,226; green,187; blue,47}, line width=0.5mm]
\tikzstyle{yellowedge2}=[-, draw={rgb,255: red,255; green,128; blue,0}, line width=0.5mm, postaction={decorate, decoration={markings, mark=at position 0.57 with {\arrow[blue]{Latex[length=2mm,width=2mm]}}}}]
\tikzstyle{fivearrows}=[-, draw=blue, line width=0.5mm, postaction={decorate, decoration={markings, mark=between positions 0.1 and 0.9 step 0.2 with {\arrow[blue]{Latex[length=2mm,width=2mm]}}}}, fill=cyan]
\tikzstyle{shadow_silver3}=[-, draw=black, fill={rgb,255: red,184; green,102; blue,255}]
\tikzstyle{shadow_silver4}=[-, draw=black, line width=0.2mm, pattern=vertical lines, pattern color=black]
\tikzstyle{shadow_thick}=[-, draw=black, fill={rgb,255: red,191; green,191; blue,191}, line width=0.35mm]
\tikzstyle{midrrows}=[-, draw=black, line width=0.2mm]
\tikzstyle{pinkedge}=[-, draw={rgb,255: red,255; green,191; blue,191}]

\newtheorem{theorem}{Theorem}[section]
\newtheorem{lemma}[theorem]{Lemma}
\newtheorem{corollary}[theorem]{Corollary}
\newtheorem{proposition}[theorem]{Proposition}

\newtheorem{claim}{Claim}
\newtheorem{observation}[theorem]{Observation}

\title{\mbox{On the Restricted Edge-Cuts of Optimal 1-Planar Graphs}}
\author{ \normalfont
  Licheng Zhang$^{1}$, Zhangdong Ouyang$^{2}$\thanks{Corresponding author}, Yuanqiu Huang$^{1}$, Guiping Wang$^{1}$\\[3pt]
\small $^1$School of Mathematics and Statistics, Hunan Normal University\\
\small $^2$School of Mathematics, Hunan First Normal University\\
\small Changsha, China\\
\small \texttt{lczmath@hunnu.edu.cn}, \texttt{oymath@163.com}, \texttt{hyqq@hunnu.edu.cn}, \texttt{wanggpmath@163.com}
}

\date{}

\begin{document}
\maketitle

\begin{abstract} The restricted edge-connectivity of a graph is the minimum size of
an edge-cut whose removal leaves every component with at least two
vertices.  In 2024, Zhang et al. showed that the restricted edge-connectivity of any optimal $1$-planar graph belongs to $\{8,10,12\}$. In this paper, we exclude $8$ as a possible value, thereby proving
that the restricted edge-connectivity is either $10$ or $12$, and both values are attainable. Furthermore, we show that the restricted edge-connectivity of a 6-connected optimal 1-planar graph equals $10$ if and only if the graph contains an edge whose two endvertices both have degree $6$. As a key ingredient, we characterize the structure of vertex-induced subgraphs on $n$ vertices with $4n-9$ edges in optimal 1-planar graphs, and use this characterization to establish a connection between restricted edge-cuts and vertex-cuts in optimal 1-planar graphs.

\end{abstract}
\noindent\textbf{Keywords:}
optimal $1$-planar graph; restricted edge-connectivity;
almost optimal $1$-planar graph


\section{Introduction}
All graphs considered in this paper are finite and simple.
A \emph{drawing} of a graph represents its vertices by distinct points in the
plane and its edges by arcs joining their corresponding endpoints. An
intersection point of the interiors of two edges is called a
\emph{crossing}. A graph is \emph{planar} if it admits a drawing without
crossings, and \emph{$1$-planar} if it admits a drawing in which each
edge is crossed at most once.  A \emph{plane graph} is a planar graph drawn in the plane without crossings.
Similarly, a \emph{$1$-plane graph} is a $1$-planar graph drawn in the plane
so that each edge is crossed at most once. 1-planar graphs were
introduced by Ringel in 1965 in connection with the simultaneous
colouring of vertices and faces of plane graphs \cite{MR0187232}.
 It is well known that a $1$-planar graph on $n\ge 3$ vertices has at most
$4n-8$ edges \cite{MR2746706}. A $1$-planar graph  is called \emph{optimal} if it has
exactly $4n-8$ edges. Optimal $1$-planar
graphs exist precisely for $n=8$ and for every $n\ge 10$, and none exist
for $n\le 7$ or $n=9$ \cite{MR2746706}. Various aspects of optimal $1$-planar graphs have been extensively
studied, including recognition~\cite{MR3741533}, matching
extendability~\cite{MR3846896,MR4606109}, Hamiltonicity~\cite{MR2993519},
edge decompositions~\cite{MR3591220}, and complete minors~\cite{MR3624608}.
 Moreover, to the best of our knowledge, several interesting problems concerning optimal 1-planar graphs remain unresolved. For example, Bekos et al. conjectured that every optimal 1-planar graph has book thickness four \cite{MR3487158}. For each $i\in\{3,4,5\}$, the
maximum possible number of copies of $K_i$ in an $n$-vertex optimal
$1$-planar graph remains unknown~\cite{MR4522420}.

In this paper, we study a classical variant of edge-cuts, namely restricted edge-cuts, in optimal 1-planar graphs. To state our main results, we introduce some terminology and explain the motivation.
Let $G$ be a connected graph with vertex set $V(G)$ and edge set $E(G)$. 
A \emph{vertex-cut} of $G$ is a set $S\subseteq V(G)$ such that $G-S$ is disconnected.
An \emph{edge-cut} of $G$ is a set $F\subseteq E(G)$ such that
$G-F$ is disconnected.
For a
non-complete graph $G$, the \emph{vertex-connectivity} $\kappa(G)$ and
the \emph{edge-connectivity} $\lambda(G)$ are the minimum sizes of a
vertex-cut and an edge-cut, respectively.  The vertex connectivity of optimal $1$-planar graphs has been studied as a tool for investigating the number of 1-planar embeddings \cite{MR2746706} and matching extendability \cite{MR3846896} of optimal $1$-planar graphs.
 Suzuki \cite{MR2746706} proved that every optimal 1-planar graph is 4-connected. This result was later strengthened by Fujisawa, Segawa, and Suzuki \cite{MR3846896}, who showed that the vertex-connectivity of any optimal 1-planar graph is either 4 or 6. It is known that optimal 1-plane graphs have no $\times$-crossings, where a $\times$-crossing is a crossing whose four endpoints induce only the two crossing edges. Thus, the linear-time algorithm of Biedl and Murali~\cite{MR4985368} for 1-plane graphs without $\times$-crossings can be used to determine the vertex-connectivity of an optimal 1-plane graph in linear time.

In 2024, Zhang, Huang, and Wang \cite{MR4665697} showed that every optimal
$1$-planar graph $G$ has edge-connectivity $6$. Moreover, they showed that every minimum
edge-cut of $G$ consists precisely of all edges incident with
some vertex of degree $6$; clearly, deleting any such edge-cut leaves an isolated 
vertex (the vertex is of degree $6$ in $G$).
Thus, it is interesting to consider what happens if we restrict the attention 
to edge-cuts that leave no isolated vertex, i.e., those that separate 
the graph into parts each having at least two vertices. 
The concept of restricted edge-cuts and edge-connectivity were  introduced by Esfahanian and Hakimi~\cite{MR0935244}.
The \emph{components} of a graph $G$ are the maximal connected subgraphs of $G$.  An edge-cut $F$ of $G$ is called a
\emph{restricted edge-cut} if  every component of $G-F$ has  at least two vertices. The
\emph{restricted edge-connectivity} of $G$, denoted by $\xi(G)$,
is the minimum cardinality of a restricted edge-cut of $G$. A restricted
edge-cut $F$ is called \emph{minimum} if $|F|=\xi(G)$.
By~\cite{MR0935244}, the restricted edge-connectivity exists for every
connected graph of order at least four that is not a star. Hence it
exists for every optimal $1$-planar graph.

Furthermore, the authors of \cite{MR4665697} established the following result on the possible values of the restricted edge-connectivity of optimal 1-planar graphs.
\begin{theorem}[\cite{MR4665697}]\label{lem:xiG}
Let G be an optimal 1-planar graph. Then 
$
    \xi(G)\in \{8,10,12\}.
$
\end{theorem}

Furthermore, the authors of \cite{MR4665697} conjectured that no optimal $1$-planar graph has $\xi(G)=8$. In this paper, we begin by establishing the following theorem, thereby confirming the conjecture.

\begin{theorem}\label{thm:main}
Let $G$ be an optimal $1$-planar graph. Then
$
    \xi(G)\in \{10,12\}.
$

\end{theorem}

Note that both values, $10$ and $12$, are attained by optimal
$1$-planar graphs~\cite{MR4665697}. 
Furthermore, we establish several structural properties of minimum restricted edge-cuts in optimal 1-planar graphs.
We first establish a close connection between 
certain special edges and  minimum restricted edge-cuts. Let $G$ be a graph. For a vertex $u\in V(G)$, let $d_G(u)$ denote the \emph{degree} of $u$ in $G$, i.e., the number of edges incident with $u$.
An edge $e = uv$ with endvertices $u$ and $v$ in  $G$ is called an \emph{$(a,b)$-edge }
if $d_G(u) = a$ and $d_G(v) = b$.
Using a result of   Hud\'ak and \v Sugerek \cite{MR2974037}, the authors of \cite{MR4665697} proved that every optimal
1-planar graph contains a \((6,6)\)-edge or a $(6,8)$-edge.
The following observation is straightforward.
\begin{observation}
  If an optimal $1$-planar graph $G$
contains a $(6,6)$-edge $uv$, then $\xi(G)=10$.
\end{observation} 
Indeed,  an optimal 1-planar graph $G$ that contains a $(6,6)$-edge $uv$ has a restricted edge-cut of
size $10$: namely,
$   \{e\in E(G): e \text{ is incident with } u \text{ or } v\}\setminus\{uv\}.
$
Thus $\xi(G)\le 10$. Together with Theorem \ref{thm:main}, this implies $\xi(G)=10$.

The above observation naturally raises the question: is the converse true? That is, does $\xi(G)=10$ imply the existence of a $(6,6)$-edge? Unfortunately, the following example shows that this is not always true.
Let $G$ be the 1-plane graph shown in Figure \ref{fig:1}. It is readily checked from the
drawing that $G$ is an optimal 1-planar graph and that $G$ contains
no $(6,6)$-edge (all the vertices of degree $6$ are shown in black).  The ten edges joining vertices in $\{x,y,z,w\}$ to vertices in $\{x',y',z',w'\}$, shown in red in the electronic version of Figure~\ref{fig:1}, form a restricted edge-cut of size 10.
Combining this with Theorem \ref{thm:main},
 we have
\(\xi(G)=10\).

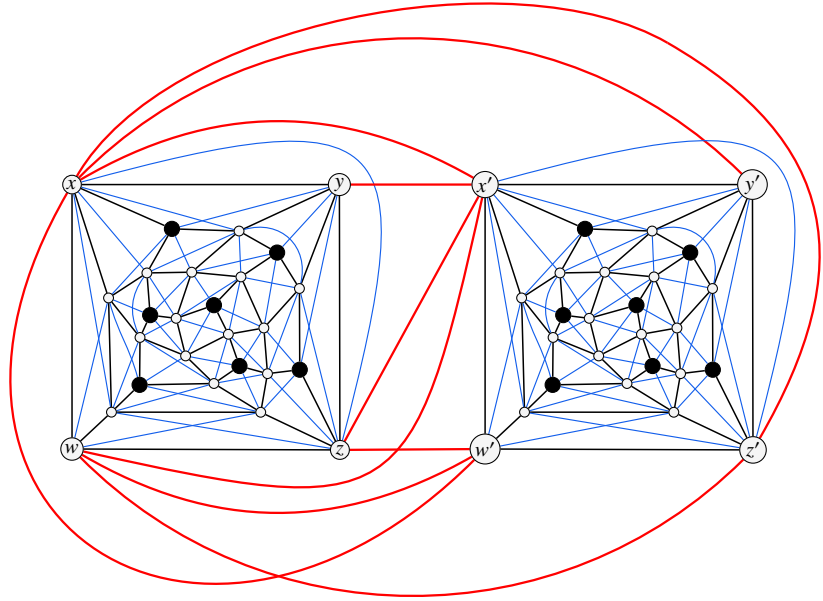
\begin{figure}[H]
\centering

\begin{tikzpicture}[scale=0.5, bezier bounding box]
	\begin{pgfonlayer}{nodelayer}
		\node [style=whitenode] (104) at (-7.248, 3.4595) {};
		\node [style=whitenode] (105) at (-6.8373, 4.6808) {};
		\node [style=whitenode] (106) at (-5.5351, 4.556) {};
		\node [style=blacknode] (107) at (-6.2495, 3.8074) {};
		\node [style=whitenode] (108) at (-5.0117, 0.9732) {};
		\node [style=whitenode] (109) at (-4.8327, 1.9949) {};
		\node [style=blacknode] (110) at (-5.5736, 2.1968) {};
		\node [style=whitenode] (111) at (-6.2492, 1.734) {};
		\node [style=whitenode] (112) at (-8.0221, 4.6607) {};
		\node [style=blacknode] (113) at (-7.9357, 3.544) {};
		\node [style=whitenode] (114) at (-8.9597, 0.9757) {};
		\node [style=blacknode] (115) at (-8.2156, 1.6987) {};
		\node [style=blacknode] (116) at (-4.5783, 5.1999) {};
		\node [style=whitenode] (117) at (-5.5849, 5.77) {};
		\node [style=whitenode] (118) at (-2.9172, -0.0238) {$z$};
		\node [style=blacknode] (119) at (-3.9798, 2.0993) {};
		\node [style=blacknode] (120) at (-7.3602, 5.8276) {};
		\node [style=whitenode] (121) at (-10.0025, 0) {$w$};
		\node [style=whitenode] (122) at (-4.9288, 3.2036) {};
		\node [style=whitenode] (123) at (-5.8674, 3.0377) {};
		\node [style=whitenode] (124) at (-8.2043, 2.953) {};
		\node [style=whitenode] (125) at (-9.0366, 3.9988) {};
		\node [style=whitenode] (126) at (-3.988, 4.2485) {};
		\node [style=whitenode] (127) at (-10.0025, 7) {$x$};
		\node [style=whitenode] (128) at (-6.9986, 2.4643) {};
		\node [style=whitenode] (129) at (-2.9313, 7) {$y$};
		\node [style=whitenode] (130) at (3.677, 3.4595) {};
		\node [style=whitenode] (131) at (4.0877, 4.6808) {};
		\node [style=whitenode] (132) at (5.3899, 4.556) {};
		\node [style=blacknode] (133) at (4.9255, 3.8074) {};
		\node [style=whitenode] (134) at (5.9133, 0.9732) {};
		\node [style=whitenode] (135) at (6.0923, 1.9949) {};
		\node [style=blacknode] (136) at (5.3514, 2.1968) {};
		\node [style=whitenode] (137) at (4.6758, 1.734) {};
		\node [style=whitenode] (138) at (2.9029, 4.6607) {};
		\node [style=blacknode] (139) at (2.9893, 3.544) {};
		\node [style=whitenode] (140) at (1.9653, 0.9757) {};
		\node [style=blacknode] (141) at (2.7094, 1.6987) {};
		\node [style=blacknode] (142) at (6.3467, 5.1999) {};
		\node [style=whitenode] (143) at (5.3401, 5.77) {};
		\node [style=whitenode] (144) at (8.0078, -0.0238) {$z'$};
		\node [style=blacknode] (145) at (6.9452, 2.0993) {};
		\node [style=blacknode] (146) at (3.5648, 5.8276) {};
		\node [style=whitenode] (147) at (0.9225, 0) {$w'$};
		\node [style=whitenode] (148) at (5.9962, 3.2036) {};
		\node [style=whitenode] (149) at (5.0576, 3.0377) {};
		\node [style=whitenode] (150) at (2.7207, 2.953) {};
		\node [style=whitenode] (151) at (1.8884, 3.9988) {};
		\node [style=whitenode] (152) at (6.937, 4.2485) {};
		\node [style=whitenode] (153) at (0.9225, 7) {$x'$};
		\node [style=whitenode] (154) at (3.9264, 2.4643) {};
		\node [style=whitenode] (155) at (7.9937, 7) {$y'$};
		\node [style=none] (156) at (5.75, 10.75) {};
	\end{pgfonlayer}
	\begin{pgfonlayer}{edgelayer}
		\draw [style=blackedge] (104) to (105);
		\draw [style=blackedge] (104) to (113);
		\draw [style=blackedge] (104) to (128);
		\draw [style=blackedge] (104) to (107);
		\draw [style=blackedge] (105) to (106);
		\draw [style=blackedge] (105) to (117);
		\draw [style=blackedge] (105) to (112);
		\draw [style=blackedge] (106) to (107);
		\draw [style=blackedge] (106) to (122);
		\draw [style=blackedge] (106) to (116);
		\draw [style=blackedge] (107) to (123);
		\draw [style=blackedge] (108) to (111);
		\draw [style=blackedge] (108) to (114);
		\draw [style=blackedge] (108) to (118);
		\draw [style=blackedge] (108) to (109);
		\draw [style=blackedge] (109) to (110);
		\draw [style=blackedge] (109) to (119);
		\draw [style=blackedge] (109) to (122);
		\draw [style=blackedge] (110) to (123);
		\draw [style=blackedge] (110) to (111);
		\draw [style=blackedge] (111) to (128);
		\draw [style=blackedge] (111) to (115);
		\draw [style=blackedge] (112) to (120);
		\draw [style=blackedge] (112) to (125);
		\draw [style=blackedge] (112) to (113);
		\draw [style=blackedge] (113) to (124);
		\draw [style=blackedge] (114) to (115);
		\draw [style=blackedge] (114) to (125);
		\draw [style=blackedge] (114) to (121);
		\draw [style=blackedge] (115) to (124);
		\draw [style=blackedge] (116) to (126);
		\draw [style=blackedge] (116) to (117);
		\draw [style=blackedge] (117) to (129);
		\draw [style=blackedge] (117) to (120);
		\draw [style=blackedge] (118) to (121);
		\draw [style=blackedge] (118) to (129);
		\draw [style=blackedge] (118) to (119);
		\draw [style=blackedge] (119) to (126);
		\draw [style=blackedge] (120) to (127);
		\draw [style=blackedge] (121) to (127);
		\draw [style=blackedge] (122) to (123);
		\draw [style=blackedge] (122) to (126);
		\draw [style=blackedge] (123) to (128);
		\draw [style=blackedge] (124) to (125);
		\draw [style=blackedge] (124) to (128);
		\draw [style=blackedge] (125) to (127);
		\draw [style=blackedge] (126) to (129);
		\draw [style=blackedge] (127) to (129);
		\draw [style=blueedge] (120) to (125);
		\draw [style=blueedge] (127) to (112);
		\draw [style=blueedge] (112) to (117);
		\draw [style=blueedge] (120) to (105);
		\draw [style=blueedge] (105) to (113);
		\draw [style=blueedge] (112) to (104);
		\draw [style=blueedge] (125) to (113);
		\draw [style=blueedge, bend right] (112) to (124);
		\draw [style=blueedge] (105) to (107);
		\draw [style=blueedge] (104) to (106);
		\draw [style=blueedge] (117) to (106);
		\draw [style=blueedge] (105) to (116);
		\draw [style=blueedge] (116) to (122);
		\draw [style=blueedge] (106) to (126);
		\draw [style=blueedge] (106) to (123);
		\draw [style=blueedge] (107) to (122);
		\draw [style=blueedge] (104) to (123);
		\draw [style=blueedge] (107) to (128);
		\draw [style=blueedge] (128) to (113);
		\draw [style=blueedge] (104) to (124);
		\draw [style=blueedge] (124) to (111);
		\draw [style=blueedge] (128) to (115);
		\draw [style=blueedge, in=285, out=105] (115) to (125);
		\draw [style=blueedge] (124) to (114);
		\draw [style=blueedge] (110) to (128);
		\draw [style=blueedge] (123) to (111);
		\draw [style=blueedge] (111) to (109);
		\draw [style=blueedge] (110) to (108);
		\draw [style=blueedge] (110) to (122);
		\draw [style=blueedge] (123) to (109);
		\draw [style=blueedge] (122) to (119);
		\draw [style=blueedge] (126) to (109);
		\draw [style=blueedge] (109) to (118);
		\draw [style=blueedge] (108) to (119);
		\draw [style=blueedge] (115) to (108);
		\draw [style=blueedge] (114) to (111);
		\draw [style=blueedge] (118) to (114);
		\draw [style=blueedge] (121) to (108);
		\draw [style=blueedge] (121) to (125);
		\draw [style=blueedge] (127) to (114);
		\draw [style=blueedge] (127) to (117);
		\draw [style=blueedge] (120) to (129);
		\draw [style=blueedge] (129) to (116);
		\draw [style=blueedge, bend left=60, looseness=1.25] (117) to (126);
		\draw [style=blueedge] (126) to (118);
		\draw [style=blueedge] (119) to (129);
		\draw [style=blackedge] (130) to (131);
		\draw [style=blackedge] (130) to (139);
		\draw [style=blackedge] (130) to (154);
		\draw [style=blackedge] (130) to (133);
		\draw [style=blackedge] (131) to (132);
		\draw [style=blackedge] (131) to (143);
		\draw [style=blackedge] (131) to (138);
		\draw [style=blackedge] (132) to (133);
		\draw [style=blackedge] (132) to (148);
		\draw [style=blackedge] (132) to (142);
		\draw [style=blackedge] (133) to (149);
		\draw [style=blackedge] (134) to (137);
		\draw [style=blackedge] (134) to (140);
		\draw [style=blackedge] (134) to (144);
		\draw [style=blackedge] (134) to (135);
		\draw [style=blackedge] (135) to (136);
		\draw [style=blackedge] (135) to (145);
		\draw [style=blackedge] (135) to (148);
		\draw [style=blackedge] (136) to (149);
		\draw [style=blackedge] (136) to (137);
		\draw [style=blackedge] (137) to (154);
		\draw [style=blackedge] (137) to (141);
		\draw [style=blackedge] (138) to (146);
		\draw [style=blackedge] (138) to (151);
		\draw [style=blackedge] (138) to (139);
		\draw [style=blackedge] (139) to (150);
		\draw [style=blackedge] (140) to (141);
		\draw [style=blackedge] (140) to (151);
		\draw [style=blackedge] (140) to (147);
		\draw [style=blackedge] (141) to (150);
		\draw [style=blackedge] (142) to (152);
		\draw [style=blackedge] (142) to (143);
		\draw [style=blackedge] (143) to (155);
		\draw [style=blackedge] (143) to (146);
		\draw [style=blackedge] (144) to (147);
		\draw [style=blackedge] (144) to (155);
		\draw [style=blackedge] (144) to (145);
		\draw [style=blackedge] (145) to (152);
		\draw [style=blackedge] (146) to (153);
		\draw [style=blackedge] (147) to (153);
		\draw [style=blackedge] (148) to (149);
		\draw [style=blackedge] (148) to (152);
		\draw [style=blackedge] (149) to (154);
		\draw [style=blackedge] (150) to (151);
		\draw [style=blackedge] (150) to (154);
		\draw [style=blackedge] (151) to (153);
		\draw [style=blackedge] (152) to (155);
		\draw [style=blackedge] (153) to (155);
		\draw [style=blueedge] (146) to (151);
		\draw [style=blueedge] (153) to (138);
		\draw [style=blueedge] (138) to (143);
		\draw [style=blueedge] (146) to (131);
		\draw [style=blueedge] (131) to (139);
		\draw [style=blueedge] (138) to (130);
		\draw [style=blueedge] (151) to (139);
		\draw [style=blueedge, bend right] (138) to (150);
		\draw [style=blueedge] (131) to (133);
		\draw [style=blueedge] (130) to (132);
		\draw [style=blueedge] (143) to (132);
		\draw [style=blueedge] (131) to (142);
		\draw [style=blueedge] (142) to (148);
		\draw [style=blueedge] (132) to (152);
		\draw [style=blueedge] (132) to (149);
		\draw [style=blueedge] (133) to (148);
		\draw [style=blueedge] (130) to (149);
		\draw [style=blueedge] (133) to (154);
		\draw [style=blueedge] (154) to (139);
		\draw [style=blueedge] (130) to (150);
		\draw [style=blueedge] (150) to (137);
		\draw [style=blueedge] (154) to (141);
		\draw [style=blueedge, in=285, out=105] (141) to (151);
		\draw [style=blueedge] (150) to (140);
		\draw [style=blueedge] (136) to (154);
		\draw [style=blueedge] (149) to (137);
		\draw [style=blueedge] (137) to (135);
		\draw [style=blueedge] (136) to (134);
		\draw [style=blueedge] (136) to (148);
		\draw [style=blueedge] (149) to (135);
		\draw [style=blueedge] (148) to (145);
		\draw [style=blueedge] (152) to (135);
		\draw [style=blueedge] (135) to (144);
		\draw [style=blueedge] (134) to (145);
		\draw [style=blueedge] (141) to (134);
		\draw [style=blueedge] (140) to (137);
		\draw [style=blueedge] (144) to (140);
		\draw [style=blueedge] (147) to (134);
		\draw [style=blueedge] (147) to (151);
		\draw [style=blueedge] (153) to (140);
		\draw [style=blueedge] (153) to (143);
		\draw [style=blueedge] (146) to (155);
		\draw [style=blueedge] (155) to (142);
		\draw [style=blueedge, bend left=60, looseness=1.25] (143) to (152);
		\draw [style=blueedge] (152) to (144);
		\draw [style=blueedge] (145) to (155);
		\draw [style=rededge, bend left] (127) to (153);
		\draw [style=rededge, bend left=45] (127) to (155);
		\draw [style=blueedge, bend left=60, looseness=2.25] (153) to (144);
		\draw [style=rededge] (153) to (118);
		\draw [style=blueedge, bend right=60, looseness=2.25] (118) to (127);
		\draw [style=rededge] (129) to (153);
		\draw [style=rededge, bend right] (121) to (147);
		\draw [style=rededge] (118) to (147);
		\draw [style=rededge, bend left=45, looseness=1.75] (153) to (121);
		\draw [style=rededge, bend right=315] (144) to (121);
		\draw [style=rededge, in=-135, out=-120, looseness=2.00] (127) to (147);
		\draw [style=rededge, in=150, out=60, looseness=0.75] (127) to (156.center);
		\draw [style=rededge, in=60, out=-30, looseness=1.25] (156.center) to (144);
	\end{pgfonlayer}
\end{tikzpicture}

\caption{An optimal 1-planar graph $G$ with $\xi(G)=10$ and without (6, 6)-edges; the 10 red edges form a minimum restricted edge-cut of $G$.}
\label{fig:1}
\end{figure}

Note that the optimal 1-plane graph in Figure~\ref{fig:1}, which provides a counterexample to the converse, has vertex-connectivity 4 ($\{x,y,z, w\}$ is a vertex-cut). Recall that every optimal 1-planar graph has vertex-connectivity either 4 or 6. We further show that the converse does hold when $G$ has vertex-connectivity 6.



\begin{theorem}\label{thm:six-connected-characterization}
Let $G$ be an optimal 1-planar graph with $\kappa(G)=6$. Then
$
        \xi(G)=10
$
if and only if $G$ contains a $(6,6)$-edge. 
\end{theorem}

The main difficulty lies in proving the ``only if'' direction of Theorem \ref{thm:six-connected-characterization}.  The main idea of the proof can be outlined as follows, although the formal proof will be organized somewhat differently.  We first investigate the structure of minimum restricted edge-cuts of size 10 in optimal $1$-planar graphs. We then argue by contradiction that if an optimal $1$-planar graph $G$ contains no $(6,6)$-edge and satisfies $\xi(G)=10$, then $G$ must contain a vertex-induced subgraph with $4n-9$ edges. The existence of such a subgraph implies that $G$ has vertex-connectivity $4$, contradicting the assumption that $\kappa(G)=6$. An important ingredient in the proof is a structural characterization of almost optimal vertex-induced subgraphs of optimal 1-planar graphs, which we establish in Section~\ref{sec:near}.

The remainder of this paper is organized as follows. Section \ref{sec:pre} introduces some notation and preliminary lemmas. We prove Theorem \ref{thm:main} in Section \ref{sec:thm1}. To prove Theorem \ref{thm:six-connected-characterization}, we introduce the notion of almost optimal 1-planar graphs and establish some of their properties in Section \ref{sec:near}. Finally, Section \ref{sec:proofTheorem2} contains the proof of Theorem \ref{thm:six-connected-characterization}.

\section{Preliminaries}\label{sec:pre}

In this section, we introduce terminology and notation and present several lemmas. Let $\mathbb R^{2}$ denote the Euclidean plane, and let $G$ be a connected plane graph. The connected components of $\mathbb R^{2}\setminus G$ are called the \emph{faces} of $G$.  For a face $\phi$ of $G$, let $\partial\phi$ denote its boundary.
The vertices and edges occurring along $\partial\phi$, counted with
multiplicity, form a closed boundary walk. If the boundary walk $\partial\phi$ is a cycle, then this cycle is
called the \emph{boundary cycle} of $\phi$.  A cycle is called a \emph{$k$-cycle} if it has $k$ edges.  A \emph{triangular face} is a face bounded by a 3-cycle, and a \emph{quadrilateral face} is  a face bounded by a 4-cycle. A \emph{quadrangulation} is a plane graph in which every face is quadrilateral.



Let $G$ be a graph.  For two disjoint vertex sets
$A,B\subseteq V(G)$, let
$
    [A,B]_G
    :=
    \{uv\in E(G): u\in A,\ v\in B\}.
$
When the graph is clear from context, we simply write $[A,B]$.
For sets $A$ and $B$, we write
$
A\subsetneq B
$
if $A$ is a \emph{proper subset} of $B$, that is,
$
A\subseteq B
$
and
$
A\ne B.
$
For $A\subseteq V(G)$, the (vertex) induced subgraph of $G$ on $A$
is denoted by $G[A]$.

For undefined terms, we refer the reader to \cite{MR2368647}.

 An edge-cut $F$ of a graph $G$ is \emph{minimal } if no proper subset of $F$ is an edge-cut. We first recall two standard facts concerning minimal edge-cuts and
minimum restricted edge-cuts.
\begin{lemma}[\cite{MR2368647}, page 62]\label{lem:twocompnents}
Let $G$ be a connected graph. If $F$ is a minimal edge-cut, then $G-F$ has
exactly two components.
\end{lemma}

\begin{lemma}[\cite{MR2365056}]\label{lem:minimal}
Let $G$ be a graph. If $F$ is a minimum restricted  edge-cut of $G$, then $F$ is a minimal edge-cut of $G$.
\end{lemma}

We shall also use the classical extremal bound for $1$-planar graphs.
\begin{lemma}[\cite{MR2297168,MR2746706}]\label{lem:4n-8}
Every 1-planar graph with $n\ge 3$ vertices has at most $4n-8$ edges. 
\end{lemma}

The following three lemmas concern optimal 1-planar graphs.
\begin{lemma}[\cite{MR3741533}]\label{lem:op_degree}
The minimum degree of every optimal 1-planar graph is six.
\end{lemma}

\begin{lemma}[\cite{MR2746706}]\label{lem:qad}
Let $G$ be an optimal 1-plane graph. Then $G$ is obtained by inserting
a pair of crossing edges to each quadrilateral face of a 3-connected quadrangulation. 
\end{lemma}

A subgraph $H$ of a graph $G$ is called a \emph{proper subgraph} of $G$
if either $V(H)$ is a proper subset of $V(G)$ or $E(H)$ is a proper
subset of $E(G)$.
The following lemma shows that the class of optimal $1$-planar graphs is not closed under taking subgraphs.
\begin{lemma}\label{lem:proper}
Any proper subgraph of an optimal $1$-planar graph $G$ is  not optimal $1$-planar.
\end{lemma}

\begin{proof}
Let $H$ be a proper subgraph of $G$, and suppose that $H$ is an optimal
$1$-planar graph. Then
$
|V(H)|\ge 8
$
by~\cite{MR2746706}. We first observe that $H$ is an induced subgraph of
$G$. Indeed, otherwise there would exist an edge of $G$ joining two
vertices of $H$ but not belonging to $H$. Adding this edge to $H$ would
produce a $1$-planar graph on $|V(H)|$ vertices with
$
4|V(H)|-7
$
edges, contradicting Lemma~\ref{lem:4n-8}. 

We next prove that $H$ is in fact the graph $G$. Since $H$ is induced, it suffices to prove that $V(G)\setminus V(H)=\emptyset$.
Fix a $1$-planar drawing $D(G)$ of $G$, and let $D(H)$ be the drawing
of $H$ obtained by restricting $D(G)$ to $H$. Let $D^{\times}(H)$ be
the planarization of $D(H)$, obtained by replacing each crossing with a new
vertex of degree $4$. The vertices inherited from $H$ are called
\emph{true vertices}, whereas the vertices corresponding to crossings
are called \emph{false vertices}.
Recall that every optimal $1$-planar graph is $4$-connected
\cite{MR2746706}, and that by Lemma \ref{lem:qad} every face of $D^{\times}$ is triangular, with
two true vertices and one false vertex on its boundary. 

\begin{claim}\label{claim:vh}
No
vertex of $V(G)\setminus V(H)$ lies in the interior of a face of
$D^{\times}(H)$.
\end{claim}
\begin{proof}
Suppose, to the contrary, that some vertex
$v\in V(G)\setminus V(H)$ lies in the interior of a face
$\Gamma$ of $D^{\times}(H)$. Let $a$ and $b$ be the two true
vertices on the boundary of $\Gamma$. Since $|V(H)|\ge 8$, choose
$w\in V(H)\setminus\{a,b\}$; then $w$ lies outside $\Gamma$.

The two boundary edges of $\Gamma$ incident with its false vertex
are segments of edges that have already been crossed, and hence
cannot be crossed again. Therefore, every path between  $v$ and $w$ must pass
through $a$ or $b$, or cross the remaining boundary edge $ab$.
Since the paths are internally vertex-disjoint and $ab$ can be
crossed at most once, there can be at most three internally
vertex-disjoint paths between $v$ and $w$. This contradicts Menger's theorem,
as $G$ is $4$-connected.
\end{proof}

By Claim \ref{claim:vh}, it follows that $V(G)\setminus V(H)=\varnothing$, and hence
$V(G)=V(H)$. Combining this with the fact that $H$ is an induced subgraph of $G$, we obtain
$H=G$, contradicting the assumption that $H$ is a proper subgraph of
$G$.
\end{proof}
\section{Proof of Theorem \ref{thm:main}}\label{sec:thm1}

%

By Theorem~\ref{lem:xiG}, 
it is enough to prove that $\xi(G)\neq 8$.
Suppose, to the contrary, that
$
    \xi(G)=8.
$
Let $F$ be a minimum restricted edge-cut of $G$. Thus, 
$
    |F|=8.
$
By Lemma \ref{lem:minimal}, $F$ is a minimal
edge-cut of $G$. By Lemma~\ref{lem:twocompnents}, $G-F$ has exactly two components,
whose vertex sets are denoted by $X$ and $Y:=V(G)\setminus X$.
Then
$
    F=[X,Y],
$
and both $G[X]$ and $G[Y]$ are connected.


\begin{claim}\label{claim:1}
$
    |X|\ge 3
$ and 
$
    |Y|\ge 3.
$
\end{claim} 
\begin{proof}
Since $F$ is a restricted edge-cut,  we have
$
    |X|\ge 2
$
and 
$
    |Y|\ge 2.
$
Suppose, to the contrary, that \(|X|=2\), where \(X=\{u,v\}\).
Since $G[X]$ is
connected, we have $uv\in E(G)$. By Lemma \ref{lem:op_degree},
$    d_G(u)\ge 6
$ and 
$
    d_G(v)\ge 6.
$
Using
$
    |[X,Y]|
    =
    \sum_{x\in X} d_G(x)-2|E(G[X])|,
$
we obtain
$
    |F|
    =
    |[\{u,v\},Y]|
    =
    d_G(u)+d_G(v)-2
    \ge 6+6-2
    =
    10,
$
which contradicts $|F|=8$. Therefore $|X|\ge 3$. By the symmetry of $X$ and $Y$, $|Y|\ge 3$. Therefore, the claim holds.
\end{proof}

\begin{claim}\label{claim:opxy}
$G[X]$ and $G[Y]$ are optimal $1$-planar graphs.
\end{claim}
\begin{proof}

Since $|V(G)|=|X|+|Y|$ and $G$ is optimal $1$-planar,
$
    |E(G)|=4(|X|+|Y|)-8.
$
Since $F=[X,Y]$ and $|F|=8$, we have
\[
\begin{aligned}
    |E(G[X])|+|E(G[Y])|
    &=
    |E(G)|-|F|  \\
    &=
    4(|X|+|Y|)-8-8 \\
    &=
    (4|X|-8)+(4|Y|-8).
\end{aligned}
\]
On the other hand,  $G[X]$ and $G[Y]$ are $1$-planar graphs. By Claim \ref{claim:1}, $|V(G[X])|\ge 3$ and $|V(G[Y])|\ge 3$. By Lemma \ref{lem:4n-8}
$
   |E(G[X])|\le 4|X|-8
$ and
$  |E(G[Y])|\le 4|Y|-8.
$
Thus $ |E(G[X])|=4|X|-8$ and $|E(G[Y])|=  4|Y|-8$, as desired.
\end{proof}
We now complete the proof. Claim~\ref{claim:opxy} implies that $G[X]$ is
optimal $1$-planar; however, since $|Y|\ge 3$, $G[X]$ is a proper
subgraph of $G$, contradicting Lemma~\ref{lem:proper}. Hence
$\xi(G)\neq 8$, and therefore $\xi(G)\in\{10,12\}$, as desired.

%
%

\section{Subgraphs of optimal 1-planar graphs}\label{sec:near}  

To prove Theorem \ref{thm:six-connected-characterization}, we study a special class of  subgraphs
of optimal $1$-planar graphs. We call a $1$-planar
graph on $n$ vertices \emph{almost optimal} if it has $4n-9$ edges.  Clearly, deleting an edge from an optimal $1$-planar graph yields an almost  optimal $1$-planar graph. But not every almost optimal $1$-planar graph is a subgraph of an
optimal $1$-planar graph; for example, $K_6$ is not.

We first establish a simple restriction on the number of vertices of an almost optimal 1-planar graph. 
 
\begin{observation}\label{ob:1}
Let $G$ be an almost optimal $1$-planar graph. Then either
$G$ is a $3$-cycle or $|V(G)|\ge 6$.
\end{observation}

\begin{proof}
Let $n:=|V(G)|$. Since $G$ is almost optimal,
$
    |E(G)|=4n-9.
$
Since $|E(G)|\ge 0$, we have $n\ge 3$. Moreover, as $G$ is simple,
$
    4n-9\le \binom{n}{2}.
$
This inequality fails for $n=4$ and $n=5$. Hence either $n=3$ or
$n\ge 6$. If $n=3$, then $|E(G)|=3$, and therefore $G$ is a
$3$-cycle.
\end{proof}
The \emph{quadrangular subgraph} of an optimal $1$-plane graph $G$,
denoted by $Q(G)$, is the subgraph consisting of all non-crossing edges
of $G$. By Lemma \ref{lem:qad}, $Q(G)$ is a $3$-connected quadrangulation. A \emph{diagonal} of a quadrilateral face is an edge joining two
opposite vertices on its boundary.

\begin{proposition}\label{prop:near-optimal}
Let $G$ be an optimal $1$-plane graph, let $Q:=Q(G)$, and let
$X\subsetneq V(G)$ with $|X|\ge 6$. Then $G[X]$ is almost optimal
if and only if $Q[X]$ satisfies the following properties.
\begin{enumerate}
    \item[(1)] $Q[X]$ is a quadrangulation;
    
    \item[(2)] $Q[X]$ has a unique face $\phi$ that is not a face of
    $Q$, while every other face of $Q[X]$ is a face of $Q$;
    
    \item[(3)] exactly one of the two diagonals of $\phi$ belongs to
    $E(G[X])$.
\end{enumerate}
\end{proposition}

\begin{proof}

Put
$
x:=|X| $ and $ m:=|E(G[X])|.
$
Throughout, every subgraph of $G$ is considered with the drawing inherited
from $G$. Let $c$ denote the number of crossings in $G[X]$.

To prove sufficiency, assume that $Q[X]$ satisfies (1)--(3).
By (1), $Q[X]$ is a quadrangulation on $x$ vertices. Hence
$
|E(Q[X])|=2x-4
$
and $Q[X]$ has $x-2$ faces.
By (2), exactly one of these faces, say $\phi$, is not a face of $Q$.
Thus the remaining $x-3$ faces are quadrilateral faces of $Q$. Since
$G$ is optimal $1$-plane, each such face contains both of its diagonal
edges in $G$, and hence contributes two edges to $G[X]$.
Finally, by (3), exactly one diagonal of $\phi$ belongs to $G[X]$.
Therefore,
$
|E(G[X])|
=(2x-4)+2(x-3)+1
=4x-9.
$
Hence $G[X]$ is almost optimal.

It remains to prove the necessity. Assume that $G[X]$ is almost optimal.

\begin{claim}\label{claim:cros}
$c\ge x-3$.
\end{claim}
\begin{proof}

By the standard lower bound on the number of crossings (see \cite[p.~10]{MR3751397}), a drawing of an $x$-vertex graph with $m$ edges and $c$ crossings satisfies $c\ge m-(3x-6).$
 Since $G[X]$ is almost optimal, $m=4x-9$, and hence
$
c\ge x-3 .     
$
\end{proof}

By Lemma~\ref{lem:qad}, each crossing of $G$ corresponds to a unique quadrilateral face of $Q$. Moreover, such a crossing remains in 
$G[X]$ if and only if all four vertices of the corresponding quadrilateral  face of $Q$
belong to $X$. Let $\mathcal F_X$ be the set of all  quadrilateral faces of $Q$ whose
four vertices belong to $X$. Then by the preceding analysis, 
$
c=|\mathcal F_X|.
$ 

We next determine the structure of $Q[X]$.
For each face $\psi\in\mathcal F_X$, let
$C_\psi$ denote its boundary cycle. Define $H$ to be the
plane subgraph of $Q[X]$ given by
\[
V(H)=\bigcup_{\psi\in\mathcal F_X}V(C_\psi)
\qquad\text{and}\qquad
E(H)=\bigcup_{\psi\in\mathcal F_X}E(C_\psi),
\]
with the embedding inherited from $Q$. We first determine
the structure of $H$ and then show that, in fact, $H=Q[X]$.

Since $x\ge 6$, Claim~\ref{claim:cros} yields
$
c\ge x-3\ge 3.
$
Thus $\mathcal F_X\neq\varnothing$, so $H$ contains the boundary of at least one quadrilateral face and hence
$
|V(H)|\ge 4.
$ Furthermore, we have the following claim.

\begin{claim}\label{claim:H-structure}
The subgraph $H$ satisfies the following properties:
\begin{enumerate}
    \item[(i)] $V(H)=X$ and $H$ is connected;
    \item[(ii)] every face of $H$ is  quadrilateral, and exactly one
    face of $H$ is not a member of $\mathcal F_X$.
\end{enumerate}
\end{claim}
\begin{proof} 
We prove (i) and (ii) simultaneously by deriving an upper bound for
$c=|\mathcal F_X|$. The equality conditions in this bound will force $H$ to satisfy properties (i) and (ii).

Since $|V(H)|>0$, let $H_1,\ldots,H_s$ be the components of $H$ where $s\ge 1$. For each $i$, let
$v_i,e_i,f_i$ be the numbers of vertices, edges and faces of $H_i$,
respectively, and let $q_i$ be the number of faces of $H_i$ that belong to
$\mathcal F_X$. Clearly $H_i$ has at least one quadrilateral face belonging to $\mathcal F_X$, and thus $v_i\ge 4$.

We next establish two bounds that will be used in the counting argument.

Since every face of the quadrangulation $Q$ is quadrilateral, it is easy to see that every cycle of $Q$ has even length; hence $Q$ is bipartite.
Thus, each subgraph $H_i$ is a bipartite plane graph on at least four vertices. A straightforward consequence of Euler’s formula is that 
$f_i\le v_i-2,$
with equality precisely when $H_i$ is a quadrangulation.

 We now show that $
q_i\le f_i-1$, that is, each component $H_i$ has at least one face that is
not a member of $\mathcal F_X$. Indeed, if every face of some $H_i$ belonged to $\mathcal F_X$, then
$H_i$ would be a quadrangulation. Hence
$
e_i=2v_i-4
$ and
$f_i=v_i-2.
$
Moreover, each face of $H_i$ is  also a quadrilateral face of $Q$ and
therefore contributes two diagonals to $G[V(H_i)]$. Thus,
\[
\begin{aligned}
|E(G[V(H_i)])|
&= e_i+2f_i\\
&=(2v_i-4)+2(v_i-2)\\
&=4v_i-8.
\end{aligned}
\]
Hence $G[V(H_i)]$ is an optimal $1$-planar graph.
Since
$
|V(H_i)|\le x<|V(G)|,
$
it is a proper subgraph of $G$, contradicting
Lemma~\ref{lem:proper}.

Therefore, we have
\vspace{-8pt}
\begin{align}
c=|\mathcal F_X|
&=\sum_{i=1}^s q_i                                      \notag\\
&\le \sum_{i=1}^s (f_i-1)                                \label{eq:cx-one-extra-face}\\
&\le \sum_{i=1}^s (v_i-3)                                \label{eq:cx-quadrangulation}\\
&=\sum_{i=1}^s v_i-3s                                    \notag\\
&\le x-3s                                                \label{eq:cx-all-vertices-used}\\
&\le x-3 .                                               \label{eq:cx-connected}
\end{align}

Together with $c\ge x-3$, the chain of inequalities above forces equality throughout.

For (i), the equality in \eqref{eq:cx-all-vertices-used} gives
$\sum_{i=1}^s v_i=x$, and hence $V(H)=X$.
 The equality in \eqref{eq:cx-connected} gives $s=1$, so
$H$ is connected. Thus (i) holds.

For (ii), 
the equality in
\eqref{eq:cx-quadrangulation} gives $f_i=v_i-2$ for the unique component, so $H$ is a
quadrangulation. Finally, the equality in \eqref{eq:cx-one-extra-face} gives
$q_i=f_i-1$, which means that exactly one face of $H$ is not a member of
$\mathcal F_X$.  Thus (ii) holds.
\end{proof}

\begin{claim}\label{claimQX}
$
     Q[X]=H.
$
\end{claim}
\begin{proof}
 Clearly, $H\subseteq Q[X]$. Suppose that
$uv\in E(Q[X])\setminus E(H)$. Since $H$ is a spanning plane subgraph of
$Q[X]$, the edge $uv$ lies in a face of $H$. By Claim
\ref{claim:H-structure}(ii), this face is a quadrilateral face. Since
$uv\notin E(H)$, its endpoints are opposite vertices of this cycle. However,
$Q$ is bipartite, so opposite vertices of a $4$-cycle belong to the same
partite set and cannot be adjacent in $Q$, a contradiction. Hence
$
       Q[X]=H.
$
\end{proof} 

We now complete the proof of the necessity by showing that $Q[X]$
satisfies conditions~(1)--(3).

Thus, by Claim~\ref{claim:H-structure}(i) and~ Claim \ref{claimQX},
$Q[X]$ is a quadrangulation, and thus (1) holds. By Claim~\ref{claim:H-structure}(ii), exactly one face of $Q[X]$ does
not belong to $\mathcal F_X$. By the definition of $\mathcal F_X$, a
face of $Q[X]$ belongs to $\mathcal F_X$ if and only if it is also a
face of $Q$. Consequently, $Q[X]$ has exactly one face that is not a
face of $Q$. Thus (2) holds.

     Note that $Q[X]$ has $2x-4$ edges and
$x-2$ faces.   Furthermore,  by Claim \ref{claim:H-structure}(ii), exactly $x-3$ of these faces belong to $\mathcal F_X$, each contributing
two diagonal edges to $G[X]$. Let $\phi$ be the unique exceptional face. Then
$
4x-9=m=(2x-4)+2(x-3)+d_\phi,
$
where $d_\phi$ is the number of diagonals of $\phi$ in $G[X]$.
Thus $d_\phi=1$. Therefore, exactly one diagonal of
$\phi$ belongs to $G[X]$. Thus (3) holds.

This completes the proof of the proposition.
\end{proof}

As an immediate consequence of
Proposition~\ref{prop:near-optimal}, an optimal $1$-plane graph
containing an almost optimal proper induced subgraph has a vertex-cut
of size~$4$, as stated in the following corollary.

\begin{corollary}\label{cor:four-cut-from-deficient-subgraph}
Let $G$ be an optimal $1$-plane graph, and let $Q:=Q(G)$.
If $X\subsetneq V(G)$, $|X|\ge 6$, and $G[X]$ is almost optimal,
then
$
\kappa(G)=4.
$
\end{corollary}

\begin{proof}
By Proposition~\ref{prop:near-optimal}, $Q[X]$ is a
quadrangulation with a unique quadrilateral face $\phi$ that is not
a face of $Q$. Let $C$ denote the boundary cycle of $\phi$.

We first show that $C$ separates $X\setminus V(C)$ from
$V(G)\setminus X$ in $Q$. Every face of $Q[X]$ other than $\phi$ is
also a face of $Q$. Therefore, no vertex of $V(Q)\setminus X$ can lie
in the interior of any face of $Q[X]$ other than $\phi$. Since
$X\subsetneq V(G)=V(Q)$, it follows that $|V(G)\setminus X|\ge 1$, and  every vertex of
$V(G)\setminus X$ lies in the region of the plane bounded by $C$
that contains $\phi$.

On the other hand, since $|X|\ge 6$ and $|V(C)|=4$, we have
$
X\setminus V(C)\ne\varnothing.
$
All vertices of $X\setminus V(C)$ lie on the other side of $C$,
because $\phi$ is a face of the plane graph $Q[X]$. Hence, by the
Jordan curve theorem, every path in $Q$ from a vertex of
$X\setminus V(C)$ to a vertex of $V(G)\setminus X$ contains a
vertex of $C$.  It remains to show that the additional edges in
$E(G)\setminus E(Q)$ do not destroy the separation, so that $V(C)$
is also a vertex-cut of $G$. By Lemma~\ref{lem:qad}, every edge of
$E(G)\setminus E(Q)$ is a diagonal drawn inside a quadrilateral face
of $Q$. Since $C$ is a cycle of $Q$, the interior of each face of
$Q$ lies entirely on one side of $C$. Consequently, no edge of
$G-V(C)$ joins a vertex of $X\setminus V(C)$ to a vertex of
$V(G)\setminus X$. Both sets $X\setminus V(C)$ and $V(G)\setminus X$ are nonempty.
Therefore, $G-V(C)$ is disconnected, and hence $V(C)$ is a
vertex-cut of $G$ of size $4$. Thus,
$
\kappa(G)\le 4.
$ 
Since every optimal $1$-planar graph is $4$-connected \cite{MR2746706}, we have
$\kappa(G)\ge 4$. Together with $\kappa(G)\le 4$, this yields
$\kappa(G)=4$.
\end{proof}

\section{Proof of Theorem \ref{thm:six-connected-characterization}}\label{sec:proofTheorem2}
In this section, we prove Theorem \ref{thm:six-connected-characterization}. We begin with the following
lemma, which describes the structure of restricted edge-cuts of size~$10$ of optimal 1-planar graphs.
\begin{lemma}
\label{lem:structure-10-cut}
Let $G$ be an optimal $1$-planar graph with $\xi(G)=10$, and let
$F$ be a minimum restricted edge-cut of $G$. Let $X$ and $Y$ be 
vertex sets of two components of $G-F$. Then exactly one of the
following holds:

\begin{enumerate}
\item[(i)] One of $G[X]$ and $G[Y]$, say $G[X]$, consists of a
single edge that is a $(6,6)$-edge of $G$.

\item [(ii)]  Both $G[X]$ and $G[Y]$ are almost optimal and satisfy
$
    |X|\ge 6
$ and 
$    |Y|\ge 6 .
$

\end{enumerate}
\end{lemma}

\begin{proof}
By Lemma \ref{lem:minimal}, $F$ is a minimal edge-cut. By Lemma
\ref{lem:twocompnents}, $G[X]$ and $G[Y]$ are exactly the  two components of $G-F$. Hence
$
F=[X,Y],
$
and both $G[X]$ and $G[Y]$ are connected. Since $F$ is a restricted
edge-cut, we have
$
|X|\ge 2
$ and 
$
|Y|\ge 2 .
$

We distinguish two cases according to the value of
$\min\{|X|,|Y|\}$; these cases correspond to {\rm (i)} and
{\rm (ii)}, respectively.

\medskip
\noindent\textbf{Case 1.} $\min\{|X|,|Y|\}=2$.

By symmetry, we may assume that $|X|=2$. Write $X=\{u,v\}$.
Since $G[X]$ is connected, $uv\in E(G)$.  Since every optimal $1$-planar graph has
minimum degree $6$, we have
$
d_G(u)\ge 6
$ and 
$
d_G(v)\ge 6 .
$
Therefore
$
10=|F|=d_G(u)+d_G(v)-2\ge 6+6-2=10.
$
Hence 
$
d_G(u)=d_G(v)=6 .
$
Thus $uv$ is a $(6,6)$-edge of $G$. Thus, (i) holds.

\medskip
\noindent\textbf{Case 2.} $
|X|\ge 3
$ and 
$
|Y|\ge 3 .
$

Put
$
 e_X:=|E(G[X])|, e_Y:=|E(G[Y])|.
$  Since $G$ is optimal $1$-planar and $|F|=10$, we have
\begin{equation}\label{eq:en-9}
\begin{aligned}
e_X+e_Y
&= |E(G)|-|F| \\
&= 4(|X|+|Y|)-8-10 \\
&= (4|X|-9)+(4|Y|-9).
\end{aligned}
\end{equation}
On the other hand, $G[X]$ and $G[Y]$ are proper induced subgraphs of $G$.
By Lemma \ref{lem:4n-8} and Lemma \ref{lem:proper}, neither of them is optimal
$1$-planar. Hence
$
e_X\le 4|X|-9
$ and
$
e_Y\le 4|Y|-9 .
$
Combining this with the equality (\ref{eq:en-9}), we obtain
$
e_X=4|X|-9
$ and 
$e_Y=4|Y|-9 .
$
Thus both $G[X]$ and $G[Y]$ are almost optimal. 

It remains to show that $|X|\ge 6$ and $|Y|\ge 6$. We prove this
for $X$; the proof for $Y$ is the same. If $|X|=3$, then
$
e_X=3 .
$
Thus
$$
|F|
=\sum_{v\in X}d_G(v)-2e_X
\ge 3\cdot 6-2\cdot 3
=12,
$$
contradicting $|F|=10$. Hence $|X|\ne 3$. Furthermore, by Observation \ref{ob:1},  we have $|X| \ge 6$. Similarly, $|Y|\ge 6$. Thus (ii) holds. This proves the lemma.
\end{proof}

\begin{proposition}\label{pro:xikappa}
Let $G$ be an optimal 1-planar graph without $(6,6)$-edges. If $\xi(G)=10$, then $\kappa(G)=4$.
\end{proposition}
\begin{proof}
Let $F$ be a minimum restricted edge-cut of $G$. Since $\xi(G)=10$, we
have $|F|=10$. Let $X$ and $Y$ be the vertex sets of the two components
of $G-F$. Since $G$ has no $(6,6)$-edge, by Lemma \ref{lem:structure-10-cut},
$
|X|\ge 6
$ and $|Y|\ge 6,
$
and both $G[X]$ and $G[Y]$ are almost optimal.
By Corollary \ref{cor:four-cut-from-deficient-subgraph},
$
\kappa(G)=4 .
$
This proves the proposition.
\end{proof}

We now prove Theorem~ \ref{thm:six-connected-characterization}.
\begin{proof}[\textbf{Proof of Theorem \ref{thm:six-connected-characterization}}]
Assume that $G$ has a $(6,6)$-edge $uv$. Let
\[
F:=[\{u,v\},V(G)\setminus\{u,v\}]
 =\{e\in E(G): e\text{ is incident with }u\text{ or }v\}
   \setminus\{uv\}.
\]
Since $d_G(u)=d_G(v)=6$, we have
$
|F|=d_G(u)+d_G(v)-2=10.
$
In $G-F$, the edge $uv$ forms a component. Moreover, every vertex
$w\in V(G)\setminus\{u,v\}$ loses at most two incident edges. By
Lemma~\ref{lem:op_degree},
$
d_{G-F}(w)\ge d_G(w)-2\ge 4.
$
Thus, $G-F$ has no isolated vertices. Therefore, $F$ is a restricted
edge-cut of $G$ of size~$10$. Combining this with Theorem~\ref{thm:main}, we obtain 
$\xi(G)=10$.

Conversely, suppose that $\xi(G)=10$. If $G$ has no $(6,6)$-edge,
then Proposition \ref{pro:xikappa} implies
$
        \kappa(G)=4,
$
contradicting the assumption that $\kappa(G)=6$. Hence $G$
contains a \((6,6)\)-edge.
\end{proof}
%
%

\section*{Declarations}
Data availability is not applicable to this article, as no datasets
were generated or analyzed during the current study.
The authors declare that they have no conflict of interest.
The work was supported by the National
Natural Science Foundation of China (Grant Nos. 12271157, 12371346).

\end{document}